\documentclass[a4paper,12pt,reqno]{amsart}
\usepackage{color}
\usepackage{txfonts}
\usepackage{bbm}
\usepackage{cases}
\usepackage{graphicx}
\usepackage{mathrsfs}
\usepackage{ulem}
\usepackage{stmaryrd}
\usepackage{amsfonts}
\usepackage{enumerate,amsmath,amssymb,amsthm,booktabs,url}
\usepackage[nobysame]{amsrefs}
\BibSpec{article}{%
+{}{\PrintAuthors} {author}
+{,}{ \textrm} {title}
+{.}{ \textit} {journal}
+{,}{ \textbf} {volume}
+{}{ \parenthesize} {date}
+{,}{ } {pages}
+{.}{ Available at arXiv:} {eprint}
+{.}{} {transition}
}
\BibSpec{book}{%
+{}{\PrintAuthors} {author}
+{,}{ \textit} {title}
+{.}{ \textrm} {series} 
+{,}{ Vol.} {volume} 
+{.}{ } {publisher}
+{,}{ } {date}
+{.}{} {transition}
}

\usepackage[colorlinks=true]{hyperref}
\hypersetup{
    linkcolor=blue,          
    citecolor=red,        
    filecolor=blue,      
    urlcolor=cyan  
}

\numberwithin{equation}{section}

\usepackage{thmtools} 

\declaretheorem[name=Theorem, parent=section]{theorem}
\declaretheorem[name=Lemma, sibling=theorem]{lemma}
\declaretheorem[name=Definition, sibling=theorem]{definition}

\declaretheorem[name=Corollary, sibling=theorem]{corollary}
\declaretheorem[name=Remark, sibling=theorem]{remark}
\declaretheorem[name=Example,sibling=theorem]{example}

\def\bt{\begin{theorem}}
\def\et{\end{theorem}}
\def\bl{\begin{lemma}}
\def\el{\end{lemma}}
\def\bd{\begin{definition}}
\def\ed{\end{definition}}
\def\bp{\begin{proposition}}
\def\ep{\end{proposition}}
\def\bc{\begin{corollary}}
\def\ec{\end{corollary}}
\def\br{\begin{remark}}
\def\er{\end{remark}}
\def\bx{\begin{example}}
\def\ex{\end{example}}

\def\p{\partial}

\def\[{{\Big[}}
\def\]{{\Big]}}
\def\<{{\langle}}
\def\>{{\rangle}}
\def\({{\Big(}}
\def\){{\Big)}}

\def\bx{{\mathbf{x}}}

\def\dif{{\mathord{{\rm d}}}}

\def\min{{\mathord{{\rm min}}}}

\def\={&\!\!=\!\!&}

\def\bB{{\mathbf B}}
\def\bC{{\mathbf C}}

\def\1{{\mathbf{1}}}

\def\cB{{\mathcal B}}

\def\cT{{\mathcal T}}

\def\mA{{\mathbb A}}
\def\mB{{\mathbb B}}
\def\mC{{\mathbb C}}

\def\mE{{\mathbb E}}

\def\mI{{\mathbb I}}

\def\mL{{\mathbb L}}

\def\mN{{\mathbb N}}

\def\mR{{\mathbb R}}
\def\mS{{\mathbb S}}

\def\sA{{\mathscr A}}
\def\sB{{\mathscr B}}

\def\sI{{\mathscr I}}

\def\sL{{\mathscr L}}

\def\sS{{\mathscr S}}

\def\geq{\geqslant}
\def\leq{\leqslant}
\def\ge{\geqslant}
\def\le{\leqslant}

\def\e{{\rm e}}

\def\eps{\varepsilon}

\def\p{\partial}

\def\[{{\Big[}}
\def\]{{\Big]}}
\def\<{{\langle}}
\def\>{{\rangle}}
\def\({{\Big(}}
\def\){{\Big)}}
\def\geq{\geqslant}
\def\leq{\leqslant}

\def\bx{{\mathbf{x}}}

\def\dif{{\mathord{{\rm d}}}}

\def\min{{\mathord{{\rm min}}}}

\def\={&\!\!=\!\!&}

\usepackage{titletoc}
\allowdisplaybreaks

\def\bpf{\begin{proof}}
\def\epf{\end{proof}}

\begin{document}
	
\title{Schauder's estimates  for nonlocal equations with singular L\'evy measures}
\date{\today}
\author{Zimo Hao, Zhen Wang and Mingyan Wu}



\address{School of Mathematics and Statistics, Wuhan University, Wuhan, Hubei 430072, P.R.China 
\\ Emails: zimohao@whu.edu.cn, wangzhen881025@163.com, mywu@whu.edu.cn}

\begin{abstract}
In this paper, we establish Schauder's  estimates  for the following  non-local equations in $ \mR^d $ :
$$
\p_tu=\sL^{(\alpha)}_{\kappa,\sigma} u+b\cdot\nabla u+f,\ u(0)=0,
$$
where $\alpha\in(1/2,2)$ and $ b:\mR_+\times\mR^d\to\mR$ is an unbounded local $\beta$-order H\"older function in $ x $ uniformly in $ t $, and $\sL^{(\alpha)}_{\kappa,\sigma}$ is a non-local $\alpha$-stable-like operator with form:
\begin{align*}
\sL^{(\alpha)}_{\kappa,\sigma}u(t,x):=\int_{\mR^d}\(u(t,x+\sigma(t,x)z)-u(t,x)-\sigma(t,x)z^{(\alpha)}\cdot\nabla u(t,x)\)\kappa(t,x,z)\nu^{(\alpha)}(\dif z),
\end{align*}
where $z^{(\alpha)}=z\1_{\alpha\in(1,2)}+z\1_{|z|\leq 1}\1_{\alpha=1}$, $ \kappa:\mR_+\times\mR^{2d}\to\mR_+ $ is bounded from above and below, $ \sigma:\mR_+\times\mR^{d}\to \mR^d\otimes \mR^d$ is a $ \gamma $-order H\"older continuous function in $ x $ uniformly in $ t $, and $ \nu^{(\alpha)} $ is a singular non-degenerate $ \alpha $-stable L\'evy measure.
\end{abstract}

\maketitle

\section{Introduction}
Let $b$ be a measurable vector-valued function on $\mR^d$, and $a$ be a measurable symmetric matrix-valued function on $\mR^d$. Denote by $\p_i$ the $i$-th partial derivative $\frac{\p}{\p x_i}$. Consider the following elliptic equation:
\begin{align}\label{in-EE}
a_{i,j}\p_i\p_ju+b_i\p_iu =f.
\end{align}
Here and below we use the Einstein summation convention. Suppose that $f$ belongs to $\bC^{\beta }$, where $\bC^\beta$ stands for the global H\"{o}lder spaces (see Subsection \ref{BHs}). Assume that there is a constant $\displaystyle \lambda >0$ such that $a$ is strictly elliptic, i.e.,
\begin{align*}
\xi_i a_{i,j}\xi_j\ge\lambda|\xi|^2,\quad\forall\xi\in\mR^d,
\end{align*}
and the H\"{o}lder norms of coefficients are all bounded by another constant $ \Lambda>0 $, i.e.,
\begin{align*}
\|a\|_{\bC^\beta}+\|b\|_{\bC^\beta}\leq \Lambda.
\end{align*}
Then, Schauder's estimates tell us that there is a constant $c=c(d,\beta,\lambda,\Lambda)>0$ such that for any solution $u\in \bC^{2+\beta}$ of  \eqref{in-EE},
\begin{align*}
\|u\|_{\bC^{2+\beta}}\leq c(\|u\|_{L^\infty}+\|f\|_{\bC^\beta}).
\end{align*}
It is well-known that Schauder's estimates play a basic role in constructing the classical solution for quasilinear PDEs, and also give an approach to show the well-posenesses of SDEs (see \cite{Wa-Zh16}, \cite{Cha-Ho-Me2}, \cite{Ha-Wu-Zh}, \cite{Ch-Zh19}, etc.). For heat equations, we can find many ways to prove such an estimate, such as \cite{Gi-Tr}, \cite{Kr96}, \cite{Kr-Pr17}, and so on.  A natural question is whether Schauder's estimates hold when we replace the local operator $a_{i,j}\p_i\p_j$ by some non-local ones?  These problems are drawn great interests recently (see \cite{Ba09}, \cite{Ba-Ka}, \cite{Do-Ki}, \cite{Im-Ji-Sh}, and \cite{Fr19}).

\vspace{0.5em}

In this paper, we consider the following equation:
\begin{align}\label{in-PDE}
\p_tu=\sL^{(\alpha)}_{\kappa,\sigma}u+b\cdot\nabla u+f,\quad u(0)=0,
\end{align}
where $ b:\mR_+\times\mR^d\to\mR^d$ is measurable, $b\cdot\nabla=b_i\p_i$, and $\sL^{(\alpha)}_{\kappa,\sigma}$ is an $\alpha$-stable-like operator with form:
\begin{align}\label{LL}
\sL^{(\alpha)}_{\kappa,\sigma}u(t,x):=\int_{\mR^d}\(u(t,x+\sigma(t,x)z)-u(t,x)-\sigma(t,x)z^{(\alpha)}\cdot\nabla u(t,x)\)\kappa(t,x,z)\nu^{(\alpha)}(\dif z),
\end{align}
where $z^{(\alpha)}=z\1_{\alpha\in(1,2)}+z\1_{|z|\leq 1}\1_{\alpha=1}$ with $ \alpha\in(0,2) $, $\sigma:\mR_+\times\mR^d\to\mR^d\otimes\mR^d$ and $ \kappa:\mR_+\times\mR^{2d}\to\mR $ are measurable, and $ \nu^{(\alpha)} $ is a non-degenerate $ \alpha $-stable L\'evy measure which can be very singular (see Subsection \ref{LE}). 
\vspace{0.5em}

Throughout this paper, we make the following assumptions on $\kappa,\sigma$ and $b$: 

\vspace{0.5em}

\begin{enumerate}[\bf (H$^{\beta}_{\kappa}$)]
    \item For some $c_0\geq 1$ and $\beta\in[0,1]$, it holds that for all $t\geq 0$ and $x,y,z\in\mR^d$,
    $$
    c_0^{-1}\leq\kappa(t,x,z)\leq c_0,\ \ |\kappa(t,x,z)-\kappa(t,y,z)|\leq c_0|x-y|^\beta.
    $$
    and in the case of $ \alpha=1 $,
    \begin{align*}
    \int_{r\leq |z| \leq R}z \kappa(t,x,z) \nu^{(\alpha)}(\dif z)=0 \ \ \hbox{for every } 0<r<R<\infty.
    \end{align*}
\end{enumerate}
\begin{enumerate}[\bf (H$^{\gamma}_{\sigma}$)]
    \item For some  $c_0\geq 1$ and $\gamma\in[0,1]$, it holds that  for all $t\geq 0$ and $x,\xi\in\mR^d$,
    $$
    c_0^{-1}|\xi|^2\leq |\sigma(t,x)\xi|^2 \leq c_0|\xi|^2,\ \ \|\sigma(t,x)-\sigma(t,y)\|\leq c_0 |x-y|^\gamma.
    $$
\end{enumerate}  
\begin{enumerate}[\bf (H$^{\beta}_{b}$)]
    \item For some $c_0\geq 1$ and $\beta\in[0,1]$, it holds that for all $t\geq 0$ and $x,y\in\mR^d$ with $ |x-y|\leq 1 $,
    $$
    |b(t,0)|\leq c_0,\ \ |b(t,x)-b(t,y)|\leq c_0 |x-y|^\beta.
    $$
\end{enumerate}
Here, $\|\cdot\|$ denotes the Hilbert-Schmidt norm of a matrix, and $|\cdot|$ denotes the Euclidean norm. Notice that $\sL^{(\alpha)}_{\kappa,\sigma}u(t,x)$ is meaningful when $u(t,\cdot)\in\bC^{\gamma}$ for some $\gamma>\alpha$.

\vspace{0.5em}

From the view point of PDEs, the drift term, instead of the diffusion term, plays a dominant role in the supercritical case $\alpha\in(0,1)$. There are several works to study Schauder estimate of PDE \eqref{in-PDE} when $\alpha\in(0,1)$, such as \cite{Pr12}, \cite{Si12}, \cite{Ch-Zh-Zh}, \cite{Ch-So-Zh}, \cite{Zh-Zh18}, \cite{Cha-Me-Pr}, and so on.  Particularly, when $\sigma$ is the identity matrix $\mI$, $\kappa\equiv1$, and $\nu^{(\alpha)}(\dif z)=1/|z|^{d+\alpha}\dif z$   with $\alpha\in(0,1)$, i.e. $\sL^{(\alpha)}_{\kappa,\sigma}=\Delta^{\frac{\alpha}{2}}$, Silvestre \cite{Si12} obtained an interior Schauder estimate under some H\"{o}lder continuous and bounded drifted coefficents. Moreover, Zhang and Zhao \cite{Zh-Zh18} studied Schauder's estimates for PDE \eqref{in-PDE} with L\'evy measure $\nu^{(\alpha)}(\dif z)=1/|z|^{d+\alpha}\dif z$. In addition, for singular L\'evy measures, Chen, Zhang and Zhao \cite{Ch-Zh-Zh} showed a Besov-type apriori estimate: for every $p>d/(\alpha+\beta-1)$, there is a constant $c>0$ such that
$$
\|u\|_{L^\infty ([0,T]; \bB^{\alpha+\beta}_{p,\infty})}\leq c \|f\|_{L^\infty ([0,T]; \bB^{\beta}_{p,\infty})},
$$
where $\bB^{\alpha+\beta}_{p,\infty}$ is the usual Besov space (see \autoref{BE} below), and the drift $b\in L^\infty([0,T];\bB_{p,\infty}^\beta)$ with $p\neq \infty$. Recently, under {\bf (H$^{\beta}_{b}$)} with $\alpha\in(\frac{1}{2},1)$ and $\alpha+\beta>1$, and $\sigma\equiv \mI$, Chaudru de Raynal, Menozzi and Priola \cite{Cha-Me-Pr} proved the following Schauder estimate for PDE \eqref{in-PDE}: 
\begin{align*}
\|u\|_{L^\infty([0,T];\bC^{\alpha+\beta})}\leq c \|f\|_{L^\infty([0,T];\bC^{\beta})}.
\end{align*}
It was not known untill now if the above Schauder's estimates hold when $\alpha\in(0,1)$ and $\sigma$ depends on $x$.
 
\vspace{0.5em}

 In the sequel, use $:=$ as a way of definition. For a Banach space $ \mB $ and $T>0$, we denote
$$
\mL_T^\infty(\mB):=L^\infty([0,T];\mB),\ \ \mL^\infty_{loc}(\mB):=\cap_{T>0} \mL_T^\infty(\mB),\ \ \mL_T^\infty:=L^\infty([0 ,T]\times \mR^d).
$$
Additionally, $a\vee b:=\max(a,b)$, $a\wedge b:=\min(a,b)$. The aim of this paper is to prove the following theorem which gives Schauder estimate for PDE \eqref{in-PDE} and the existence of classical solutions (see \autoref{DEF} below).


\begin{theorem}\label{SEA}
    Suppose that $ \alpha\in(1/2,2) $, $\gamma\in(\frac{1-\alpha}{\alpha}\vee 0,1]$, $\beta\in((1-\alpha)\vee0,(\alpha\wedge1)\gamma)$, and $\alpha+\beta\notin\mN$. Under {\bf (H$^{\beta}_{\kappa}$)}, {\bf (H$^{\gamma}_{\sigma}$)}, and {\bf (H$^{\beta}_{b}$)}, for any $f \in \mL_{loc} (\bC^\beta)$,  there is a unique classical solution $u$ in the sense of \autoref{DEF} such that for any $ T>0 $ and some constant $ c=c(T,c_0,d,\alpha,\beta,\gamma)>0 $, 
    \begin{align}\label{eq:U01}
    \|u\|_{\mL_T^\infty(\bC^{\alpha+\beta})}\leq c \|f\|_{\mL_T^\infty(\bC^{\beta})},\ \ \|u\|_{\mL^\infty_T}\leq T\|f\|_{\mL^\infty_T}.
    \end{align}
\end{theorem}



We point out that there are few results of heat kernel estimates for the operator $\p_t-\sL^{(\alpha)}_{\kappa,\sigma}$ when $\nu^{(\alpha)}(\dif z)\ne 1/|z|^{d+\alpha}\dif z$ and $\sigma$ is not a constant. Hence, it seems to be quite difficult to obtain Schauder's estimates as stated in \autoref{SEA} by using methods from \cite{Cha-Me-Pr}.  A key ingredient in our approach is the use of Littlewood-Paley's theory.


\begin{remark}
Notice that $b(x)=x$ satisfies the condition {\bf (H$^{\beta}_{b}$)} for any $\beta\in[0,1]$. Hence, \autoref{SEA} covers some unbounded drift cases.
\end{remark}

\br
In the \autoref{SEA}, $\alpha$ is required to be greater than $1/2$ due to some moment problems (see \autoref{R38} ). This restriction also appears in \cite{Cha-Me-Pr}. An open problem is to drop the restriction $\alpha>1/2$. If $\sL^{(\alpha)}_{\kappa,\sigma}=\Delta^{\alpha/2}$, i.e. $\sigma\equiv\mI$, $\kappa\equiv1$, we can drop the restriction $\alpha>1/2$, and then obtain Schauder's estimates for $\alpha\in(0,2)$ in our method( see \autoref{R310} ).
\er

This paper is organized as follows: In Section \ref{Pre}, we introduce some basic function spaces, and present the estimates of Littlewood-Paley's types for heat kernels of nonlocal operators with constant coefficients (see \autoref{CRU} below). In Section \ref{SSS}, we show the the maximum principle \autoref{MM}, and prove Schauder's estimates \autoref{SFA} for PDE \eqref{in-PDE} by freezing coefficients along the characterization curve. In Section \ref{EX}, through the continuity method, we apply the apriori estimate \autoref{SFA} to show the main result \autoref{SEA}.

\vspace{0.5em}

Throughout this paper, we shall use the following conventions and notations:
\begin{itemize}
\item
We use $A\lesssim B$ to denote $A\le c B$ for some unimportant constant $c>0$.
\item $ \mN_0:=\mN \cup \{0\}$, $ \mR_+:=[0,\infty) $, and for $R\geq0$, we shall denote $B_R:=\{x\in\mR^d: |x|<R\}$.
\item
If $\Omega$ is a domain in $\mR^d$, for any $p\in[1,\infty]$, we denote by  $L^p(\Omega)$ the space of $p$-summable function on $\Omega$ . Denote $L^p:=L^p(\mR^d)$ with the norm denoted $\|\cdot\|_p$.
\item
For two operators $\sA_1,\sA_2$, we use $[\sA_1,\sA_2]:=\sA_1\sA_2-\sA_2\sA_1$ to denote their commutator. 
\end{itemize}

\section{Preliminaries}\label{Pre}

\subsection{H\"{o}lder spaces and Besov spaces}\label{BHs}

We first introduce the H\"older spaces. For $h\in\mR^d$ and $f:\mR^d\to\mR$, the first order difference operator is defined by
$$
\delta_hf(x):=f(x+h)-f(x).
$$
For $0<\beta\notin \mN_0$, if $\Omega$ is a domain in $\mR^d$, let $\bC^{\beta}(\Omega)$ be the usual $\beta$-order H\"older space on $\Omega$ consisting of all functions $f:\Omega\to\mR$ with
$$
\|f\|_{\bC^\beta(\Omega)}:=\|f\|_{L^\infty(\Omega)}+\cdots+\|\nabla^{[\beta]} f\|_{L^\infty(\Omega)}+[\nabla^{[\beta]}  f]_{\bC^{\beta-[\beta]}(\Omega)}<\infty,
$$
where $[\beta]$ denotes the greatest integer less than $\beta$, and $\nabla^j$ stands for the $j$-order gradient, and
$$
[f]_{\bC^\gamma(\Omega)}:=\sup_{\substack{x,x+h\in\Omega\\ h\ne0}}|\delta_h f(x)|/|h|^{\gamma},\ \gamma\in[0,1).
$$ 
Denote $\bC^\beta:=\bC^\beta(\mR^d)$ and $[\cdot]_{\bC^\beta}:=[\cdot]_{\bC^\beta(\mR^d)}$.
We introduce another notation:
$$
[f]_{\mC^\gamma}:=\sup_{0<|h|\leq 1}\|\delta_h f\|_\infty/|h|^{\gamma},\ \gamma\in[0,1).
$$
For any integer $n\geq 1$, define $\bC^n$ be the set of  $n$-order continuous differentiable functions on $\mR^d$ with  
$$
\|f\|_{\bC^n}:= \sum_{k=1}^n \|\nabla^n f\|_{L^\infty(\mR^d)}< \infty.
$$
If $f$ belongs to $\bC^1$, then it is Lipschtiz.

\begin{remark}
	For $0<s<1$, note that the set consisting of all functions whose $\mC^s$-seminorms are finte is bigger than $\bC^s$-seminorms. For example, let $f(x)=x$,  then $[f]_{\mC^s}<\infty$ but $[f]_{\bC^s}=\infty$. This fact tells us that for some unbounded functions, their $ \mC^s $-seminorms can be finite.
\end{remark}

The following result is simple but important (see \cite{Wa-Zh16}*{Lemma2.1}). 

\begin{lemma}\label{HB}
    If $0<s<1$, then for any $x,y\in\mR^d$ with $|x-y|>1$,
    \begin{align*}
    |f(x)-f(y)|\le 2[f]_{\mC^s}|x-y|.
    \end{align*}
\end{lemma}

In the sequel, let $\chi:\mR^d\to\mR$ be a smooth function with
\begin{align}\label{CHI}
\chi(x)=
\begin{cases}
1 \quad\text{if $|x|\leq\frac{1}{4}$};\\
0\quad\text{if $|x|>\frac{1}{2}$}.
\end{cases}
\end{align}
Notice that $[f]_{\mC^s}<\infty$ is the local property of $f$ and $[f]_{\bC^s}<\infty$ is the global property. The following results are useful and their proofs are straightforward and elementary.
\begin{lemma}\label{CutF}
    For fixed $x_0\in \mR^d$, define
    $
    \tilde{f}(x):=(f(x)-f(x_0))\chi(x-x_0),\,\forall x\in\mR^d.
    $ 
    Then, for any $s\in(0,1)$, there exists a constant $ c=c(s,\chi)>0 $ such that
    $$ \|\tilde{f}\|_{\bC^s}\leq c[f]_{\mC^{s}}. 
    $$
\end{lemma}

\begin{lemma}\label{FuHe}
	Assume that $ \gamma\in(0,1] $ and $ \beta\in(0,\gamma) $.
	Let $\phi:\mR^d\to\mR^d$ with $[\phi]_{\bC^\gamma}<\infty $, and let $f\in\bC^{\beta/\gamma}$. Then, $g(\cdot):=f(\cdot+\phi(\cdot))\in\bC^{\beta}$ with
	\begin{align*}
	\|g\|_{\bC^{\beta}}\leq 2\(1+[\phi]_{\bC^\gamma}^{\beta/\gamma}\)\|f\|_{\bC^{\beta/\gamma}}.
	\end{align*}
\end{lemma}
Next we are going to introduce the Besov spaces. Let $\sS(\mR^d)$ be the Schwartz space of all rapidly decreasing functions on $\mR^d$, and $\sS'(\mR^d)$ be the dual space of $\sS(\mR^d)$ 
called Schwartz generalized function (or tempered distribution) space. For any $f\in\sS(\mR^d)$, the Fourier transform $\hat f$ and the inverse Fourier transform $\check f$ are defined by
\begin{align*}
    \hat f(\xi)&:=(2 \pi)^{-d/2}\int_{\mR^d} \e^{-i\xi\cdot x}f(x)\dif x, \quad\xi\in\mR^d,\\
    \check f(x)&:=(2 \pi)^{-d/2} \int_{\mR^d} \e^{i\xi\cdot x}f(\xi)\dif\xi, \quad x\in\mR^d.
\end{align*}
For any $f\in\sS'(\mR^d)$, the Fourier transform $\hat f$ and the inverse Fourier transform $\check f$ are defined by
$$
\<\hat f,\varphi\>:=\<f,\hat\varphi\>,\quad\quad\<\check f,\varphi\>:=\<f,\check\varphi\>,\quad \forall \varphi\in\sS(\mR^d).
$$
Let $\phi_0$ be a radial $C^\infty$-function on $\mR^d$ with 
$$
\phi_0(\xi)=1\ \mbox{ for } \ |\xi|\leq 1\ \mbox{ and }\ \phi_0(\xi)=0\ \mbox{ for } \ |\xi|>2.
$$
For $\xi=(\xi_1,\cdots,\xi_n)\in\mR^d$ and $j\in\mN$, define 
$$
\phi_j(\xi):=\phi_0(2^{- j}\xi)-\phi_0(2^{-(j-1)}\xi).
$$
It is easy to see that for $j\in\mN$, $\phi_j(\xi)=\phi_1(2^{-(j-1)}\xi)\geq 0$ and
$$
{\rm supp}\phi_j\subset \{\xi\in \mR^d \mid 2^{j-1}\le|\xi|\le2^{j+1}\},\ \  \sum_{j=0}^{k}\phi_j(\xi)=\phi_0(2^{-k}\xi)\to 1,\ \ k\to\infty.
$$
In particular, if $|j-j'|\geq 2$, then
$$
\mathrm{supp}\phi_1(2^{-j}\cdot)\cap\mathrm{supp}\phi_1(2^{-j'}\cdot)=\emptyset.
$$
From now on we shall fix such $\phi_0$ and $\phi_1$.  For $j\in \mN_0$, the block operator $\Delta_j$ is defined on $\sS'(\mR^d)$ by
\begin{align}\label{BO}
    \Delta_jf(x):=(\phi_j\hat f)\check{\,\,}(x)=\check\phi_j* f(x)=2^{(j-1)d}\int_{\mR^d}\check\phi_1(2^{j-1}y) f(x-y)\dif y.
\end{align}

\br
For $j\in\mN_0$, by definitions it is easy to see that
\begin{align}\label{DE}
    \Delta_j=\Delta_j\widetilde\Delta_j,\ \mbox{ where }\ \widetilde\Delta_j:=\Delta_{j-1}+\Delta_{j}+\Delta_{j+1}\mbox{ with } \Delta_{-1}\equiv 0,
\end{align}
and $\Delta_j$ is symmetric in the sense that
$
\<\Delta_j f,g\>=\< f,\Delta_jg\>.
$
\er

Here is the definition for the Besov spaces.

\begin{definition}[Besov spaces]\label{BE}
    For any $s\in\mR$ and $p,q \in [1,\infty]$, the Besov space $\bB^s_{p,q}$ is defined as the set of all $f\in\sS'(\mR^d)$ such that
    $$
    \|f\|_{\bB^s_{p,q}}: = \(\sum_{j =0}^\infty 2^{2^{jqs}}  \|\Delta_j f \|_{p}^q\)^{1/q} < \infty.
    $$
    When $q=\infty$, it is in the following sense
    $$
    \|f\|_{\bB^s_{p,\infty}}:=\sup_{j\in \mN_0}2^{js}\|\Delta_j f\|_p<\infty.
    $$
\end{definition}

Recall the  Bernstein's inequality( \cite{Ba-Ch-Da}*{Lemma 2.1}). 

\begin{lemma}[Bernstein's inequality]
    For any $k=0,1,2,\cdots$, there is a constant $c=c(k,d)>0$ such that for all $j\geq 0$,
    \begin{align}\label{Bernstein}
    \|\nabla^k\Delta_j f\|_\infty\leq c2^{kj}\|\Delta_jf\|_\infty.
    \end{align}
\end{lemma}

\begin{remark}\label{DF}It is well-known that for any  $0<s\notin \mN$ and $n\in\mN$,
$$
\|f\|_{\bB^s_{\infty,\infty}}\asymp\|f\|_{\bC^s},\quad\|f\|_{\bB^n_{\infty,\infty}}\lesssim\|f\|_{\bC^n}.
$$
The proof can be found in \cites{Tr92} or \cite{Ba-Ch-Da}.
\end{remark}

We also need the following interpolation inequality \cite{Be-Lo}*{Theorem 6.4.5-(3)}.

\begin{lemma}\label{INTER}
	Let $ \mA $ be a Banach space and $ \beta_1<\beta_2 $ be two positive noninteger numbers. If $\cT$ is a bounded linear operator from $\bC^{\beta_2}$ to $\mA$, then there is a constant $ c=c(\beta_1,\beta_2)>0 $ such that
	\begin{align*}
	\|\cT\|_{\cB(\bC^\beta,\mA)}
	\leq c \|\cT\|_{\cB(\bC^{\beta_1},\mA)}^\theta\|\cT\|_{\cB(\bC^{\beta_2},\mA)}^{1-\theta},
	\end{align*}
	where	$\beta=\theta\beta_1+(1-\theta)\beta_2 \notin \mN_0$ for some $\theta\in[0,1]$.
\end{lemma}

\subsection{L\'evy measures and heat kernel estimates}\label{LE}
We call a measure $\nu$ on $\mR^d$ a L\'evy measure if 
$$
\nu(\{0\})=0,\ \ \int_{\mR^d}\big(1\wedge|x|^2\big)\nu(\dif x)<+\infty.
$$
In particular, for $\alpha\in(0,2)$, we say a L\'evy measure $\nu^{(\alpha)}$ is $\alpha$-stable if it has form
\begin{align}\label{aStable}
\nu^{(\alpha)}(A)=\int^\infty_0\left(\int_{\mS^{d-1}}\frac{1_A (r\theta)\Sigma(\dif\theta)}{r^{1+\alpha}}\right)\dif r,\quad A\in\sB(\mR^d),
\end{align} 
where $\Sigma$ is a finite measure over the unit sphere $\mS^{d-1}$ (called spherical measure of $\nu^{(\alpha)}$). Notice that, for any $\gamma_1>\alpha>\gamma_2\geq 0$, 
\begin{align}\label{z}
\int_{\mR^d} {(|z|^{\gamma_1}\wedge |z|^{\gamma_2})}\nu^{(\alpha)}(\dif z)<\infty.
\end{align}
 One says that  an $\alpha$-stable measure  $\nu^{(\alpha)}$ is  non-degenerate if
\begin{align}\label{nonDe}
\int_{\mS^{d-1}}|\theta_0\cdot\theta|^\alpha\Sigma(\dif\theta)>0 \quad \hbox{for every } \theta_0\in\mS^{d-1}.   
\end{align}		

\begin{example}[Standard $\alpha$-stable measures]
	If $\Sigma$ is  the uniform measure on unit sphere $\mS^{d-1}$, then $ \nu^{(\alpha)} $ is the standard or strict $\alpha$-stable L\'{e}vy measure and
	$$
	\nu^{(\alpha)} (\dif y)=\frac{\dif y}{|y|^{d+\alpha}}.
	$$
	In this case, 
	\begin{align*}
       \widehat{\sL_{1,\mI}^{(\alpha)}f}(\xi)=-|\xi|^\alpha\hat{f}(\xi),\ \ \forall \xi\in \mR^d,
       \end{align*}
	where $\sL_{\kappa,\sigma}^{(\alpha)}$ is defined by \eqref{LL} and $\mI$ is the identity matrix.
\end{example}

\begin{example}[Cylindrical $\alpha$-stable measures]
	If $\Sigma= \sum_{k=1}^d \delta_{e_k}$, where
	$\delta_{e_k}$ is the Dirac measure at the $e_k = (0,\cdots,0,1_{k_{th}},0,\cdots,0)$, then 
$$
\nu^{(\alpha)}(\dif {x})=\sum_{k=1}^n \delta_{0}(\dif x_1)\cdots\delta_{0}(\dif x_{k-1})  \frac{\dif x_k}{|x_k|^{1+\alpha}}\delta_{0}(\dif x_{k+1})\cdots\delta_{0}(\dif x_d),
$$
where $\delta_{0}$ is the Dirac measure at the zero. Such measure is called the cylindrical L\'evy measure. Moreover, 
\begin{align*}
 \widehat{\sL_{1,\mI}^{(\alpha)}f}(\xi)=-\sum_{i=1}^d|\xi_i|^\alpha\hat{f}(\xi),\ \ \forall \xi =(\xi_1,..,\xi_d) \in \mR^d.
\end{align*}
Notice that $|\xi|^\alpha$ is not smooth at origin, and $\sum_{i=1}^d|\xi_i|^\alpha$ is not smooth on all axes $\cup_{i=1}^d\{\xi_i=0\}$.
In other words, $\sum_{i=1}^d|\xi_i|^\alpha$ is more singular than $|\xi|^\alpha$.
\end{example}  

Fix $\alpha\in(0,2)$.  Let $\nu^{(\alpha)}$ be a non-degenerate $\alpha$-stable measure, $ \kappa:\mR_+\times\mR^d\to\mR_+ $ and $ \sigma:\mR_+\to\mR^d\otimes\mR^d$ be measurable functions satisfying the following assumptions:
\begin{align}\label{con0}
	c_0^{-1}\leq \kappa(t,z)\leq c_0,\, c_0^{-1}|\xi|\le|\sigma(t)\xi|\le c_0|\xi|,\quad \forall (t,z,\xi)\in\mR_+\times\mR^d\times\mR^d,
\end{align} 
for some constant $c_0\ge1$ and in the case of $ \alpha=1 $,
\begin{align}\label{con1}
	\int_{r\leq |z| \leq R}z \kappa(t,z) \nu^{(\alpha)}(\dif z)=0 \ \ \hbox{for every } 0<r<R<\infty.
\end{align} 
Let $ N(\dif t,\dif z) $ be the Possion random measure with intensity measure $ \kappa(t,z)\nu^{(\alpha)}(\dif z)\dif t $. For $ 0\leq s\leq t $, define
\begin{align}\label{L}
L_{s,t}^\kappa:=\int_{s}^t \int_{\mR^d}z\tilde N(\dif r,\dif z)+\int_{s}^t \int_{\mR^d} (z-z^{(\alpha)})\kappa(r,z)\nu^{(\alpha)}(\dif z)\dif r,
\end{align}
where $ \tilde{N}(\dif r,\dif z):= N(\dif r,\dif z)-\kappa(r,z)\nu^{(\alpha)}(\dif z)\dif r $ is the compensated Poisson random measure and $$
z^{(\alpha)}:=z\1_{\alpha\in(1,2)}+z\1_{|z|\leq 1}\1_{\alpha=1}.
$$
Precisely,
\begin{align*}
	L_{s,t}^\kappa=
	\begin{cases}
	\vspace{0.6em}
	\int_{s}^t \int_{\mR^d}z N(\dif r,\dif z),&\ \ \hbox{if } \alpha\in(0,1);\\
	\vspace{0.6em}
	\int_{s}^t \int_{|z|\leq 1}z\tilde{N}(\dif r,\dif z) +\int_{0}^t \int_{|z|> 1} zN(\dif r,\dif z),&\ \ \hbox{if } \alpha=1;\\
	\int_{s}^t \int_{\mR^d}z\tilde N(\dif r,\dif z),&\ \ \hbox{if } \alpha\in(1,2).
	\end{cases}
\end{align*}
Next, we consider the following process: 
\begin{align}\label{X-def}
	X^{\kappa,\sigma}_{s,t}:=\int_s^t \sigma(r) \dif  L_{s,r}^\kappa,\ \ 0<s<t<\infty.
\end{align}
By the same argument as in \cite{Ch-Ha-Zh}, we have the crucial lemma in this paper

\begin{lemma}\label{CRU}
    Let $\alpha\in (0,2)$. Under \eqref{con0} and \eqref{con1}, the random variable $X_{s,t}^{\kappa,\sigma}$ defined by \eqref{X-def} has a smooth density $p_{s,t}^{\kappa,\sigma}$. Furthermore, for any $T>0$, $\beta\in[0,\alpha)$, and $ n\in\mN_0$, there is a constant $c=c(c_0,\alpha,\nu^{(\alpha)},\beta,T,d)$ such that for any $s,t\in[0,T]$ and $ j\in \mN $,
        \begin{align}\label{CruF}
    \int_0^t\int_{\mR^ d}|x|^\beta|\Delta_j p_{s,t}^{\kappa,\sigma}(x)|\dif x\dif s \leq c 2^{-(\alpha+\beta)j}.
    \end{align}
\end{lemma}

\section{Schauder's estimates for nonlocal equations}\label{SSS}

In this section, we show Schauder's estimates for nonlocal  equations: 
\begin{align}\label{FE}
    \partial_tu=\sL^{(\alpha)}_{\kappa,\sigma} u+b\cdot\nabla u+f,\ \ u(0)=0,
\end{align}
where $ b :\mR^+\times \mR^d \to \mR^d $ is a measurable function and $\sL^{(\alpha)}_{\kappa,\sigma}$ is defined by \eqref{LL} with $ \alpha\in(0,2) $:
$$
\sL^{(\alpha)}_{\kappa,\sigma}u(t,x):=\int_{\mR^d}\(u(t,x+\sigma(t,x)z)-u(t,x)-\sigma(t,x)z^{(\alpha)}\cdot\nabla u(t,x)\)\kappa(t,x,z)\nu^{(\alpha)}(\dif z). 
$$
Throughout this section, we assume that $ \kappa,\sigma $ and $ b $ satisfy, respectively,  conditions {\bf (H$^{\beta}_{\kappa}$)}, {\bf (H$^{\gamma}_{\sigma}$)} and {\bf (H$^{\beta}_{b}$)} .

\begin{definition}[Classical solutions]\label{DEF}
	We call a bounded continuous function $u$ 
	defined on $\mR_+\times\mR^{d}$ a classical solution of PDE \eqref{FE} if for some $\eps\in(0,1)$,
	$$
	u\in \(\cap_{M>0}C(\mR_+;\bC^{(\alpha\vee 1)+\eps}(B_M))\)\cap\mL^\infty_{loc}(\bC^{(\alpha\vee 1)+\eps})
	$$
 and for all $ (t,x)\in[0,\infty)\times \mR^d $,
    \begin{align}\label{cla}
    u(t,x)=\int^t_0\Big(\sL^{(\alpha)}_{\kappa,\sigma} u+b\cdot\nabla u+f\Big)(s,x)\dif s.
	\end{align}
\end{definition}
\br
Note that under the conditions {\bf (H$^{0}_{\kappa}$)} and {\bf (H$^{0}_{\sigma}$)}, $\sL^{(\alpha)}_{\kappa,\sigma}u(t,x)$ and $b\cdot\nabla u(t,x)$ is pointwisely well defined for any $u\in\mL_{loc}^\infty(\bC^\gamma)$ with $\gamma>\alpha\vee1$. Hence, the classical solution is well-defined.
\er

We have the following maximum principle for classical solutions. 

\begin{lemma}[Maximum principle]\label{MM}
	Assume that $\sigma(t,x)$ and $\kappa(t,x,z)\geq 0$ are bounded measurable functions.  Let  $b(t,x)$ be a measurable function and bounded in $ \mR_+ $ for fixed $x\in\mR^d$. Then, for any  $ T>0 $ and classical solution $u$ of PDE \eqref{FE} in the sense of \autoref{DEF}, it holds that 
	\begin{align*}
	\|u\|_{\mL^\infty_T}\leq T\|f\|_{\mL^\infty_T}.
	\end{align*}		
\end{lemma} 
\begin{proof}
    Define
	$$
	\bar u(t,x):=-u(t,x)+\int_0^t\|f(s,\cdot)\|_{\infty}\dif s\quad\hbox{and}\quad \uline{u}(t,x):=-u(t,x)-\int_0^t\|f(s,\cdot)\|_{\infty}\dif s.
	$$
	By \eqref{cla}, it is easy to see that for Lebesgue almost all $ t>0 $,
	\begin{align*}
	\partial_t\bar u-\sL_{\kappa,\sigma}^{(\alpha)}\bar u-b\cdot\nabla\bar u\ge0\quad \hbox{and}\quad\partial_t\uline u-\sL_{\kappa,\sigma}^{(\alpha)}\uline u-b\cdot\nabla\uline u\leq0,
	\end{align*}
	where $\lim_{t\to0}\bar u(t,x)=\lim_{t\to0}\uline u(t,x)=0$. Notice that the form of $\nu^{(\alpha)}$ does not affect the result of \cite{CHXZ}*{Theorem 6.1}. Thus, by \cite{CHXZ}*{Theorem 6.1} and \autoref{HB}, we have
	\begin{align*}
	\uline u(t,x)\leq0\leq\bar u(t,x)
	\end{align*}
	which implies that
	\begin{align*}
	|u(t,x)|\le\int_0^t\|f(s,\cdot)\|_{\infty}\dif s\leq T\|f\|_{\mL_T^\infty}.
	\end{align*}	
The desired estimate is proved.	
\end{proof}

Our  goal of this section is to prove the following Schauder's apriori estimates. 

\begin{theorem}[Schauder's estimates]\label{SFA}
    Let $ \alpha\in(1/2,2) $, $\gamma\in(\frac{1-\alpha}{\alpha}\vee0,1]$ and $\beta\in((1-\alpha)\vee0,(\alpha\wedge1)\gamma)$ with $\alpha+\beta\notin\mN$. Under the conditions {\bf (H$^{\beta}_{\kappa}$)}, {\bf (H$^{\gamma}_{\sigma}$)}, and {\bf (H$^{\beta}_{b}$)}, for any $T>0$ and $f \in \mL_T^\infty(\bC^{\beta})$, there is a constant $ c=c(T,c_0,d,\alpha,\beta,\gamma)>0 $ such that for any classical solution $u$ of PDE \eqref{FE},
    \begin{align}\label{eq:U01}
    \|u\|_{\mL_T^\infty(\bC^{\alpha+\beta})}\leq c \|f\|_{\mL_T^\infty(\bC^{\beta})}.
    \end{align}
\end{theorem}

To prove this theorem, we use the perturbation argument by freezing coefficients along the characterization curve as showed in \cite{Ha-Wu-Zh}. We need the following well-known fact from ODE, whose proof can be found in \cite{Ha-Wu-Zh}*{Lemma 6.5}. 
\begin{lemma}\label{HWZ}
	Let $b:\mR_+\times\mR^{d}\to\mR^{d}$ be a  measurable vector field. Suppose that for each $t>0$, $x\mapsto b(t,x)$ is continuous and
	there is a constant $c>0$ such that for all $(t,x)\in\mR_+\times\mR^d$,
	$$
	|b(t,x)|\leq c(1+|x|).
	$$
	Then, for each $x\in\mR^{d}$, there is a global solution $\theta_t$ to the following ODE:
	$$
	\dot\theta_t=b(t,\theta_t),\ \ \theta_0=x.
	$$
	Moreover, if we denote by $\sS_{x}:=\{\theta_\cdot: \theta_0=x\}$ the set of all solutions with starting point $x$, then for each $T>0$,
	\begin{align*}
	\cup_{x\in\mR^d}\cup_{\theta_\cdot\in\sS_{x}}\{\theta_T\}=\mR^d.
	\end{align*}
\end{lemma}

\subsection{Bounded drift case}\label{UnS}
In this subsection, assuming $b\in\mL_{loc}^\infty(\bC^\beta)$, we prove the following apriori estimate.

\begin{theorem}\label{CO}
	Let $ \alpha\in(1/2,2) $, $\gamma\in(\frac{1-\alpha}{\alpha}\vee0,1]$ and $\beta\in((1-\alpha)\vee0,(\alpha\wedge1)\gamma)$ with $\alpha+\beta\notin\mN$. Under the conditions {\bf (H$^{\beta}_{\kappa}$)}, {\bf (H$^{\gamma}_{\sigma}$)}, and $b\in\mL^\infty_{loc}(\bC^\beta)$, for any $T>0$ and $f \in \mL_T^\infty(\bC^{\beta})$, there is a constant $ c=c(T,\|b\|_{\mL^\infty_{T}(\bC^\beta)},d,\alpha,\beta,\gamma)>0 $ such that for any classical solution $ u $ of PDE \eqref{FE},
	\begin{align}\label{co}
	\|u\|_{\mL_T^\infty(\bC^{\alpha+\beta})}\leq c \|f\|_{\mL_T^\infty(\bC^{\beta})}.
	\end{align}
\end{theorem}

\vspace{0.5em}

Fix $x_0\in\mR^{d}$. Let $\theta_t$ solve the following ODE in $\mR^{d}$:
$$
\dot\theta_t=-b(t,\theta_t), \ \ \theta_0=x_0.
$$
Define
$$
\tilde u(t,x):=u(t,x+\theta_t),\ \ \tilde f(t,x):=f(t,x+\theta_t),\ \ \tilde \sigma(t,x):=\sigma (t,x+\theta_t),
$$
$$
\tilde \kappa(t, x,z):=\kappa(t,x+\theta_t,z),\ \ \tilde \sigma_0(t):=\tilde \sigma(t,0),\ \ \tilde \kappa_0(t,z):=\tilde \kappa(t,0,z),
$$
and 
\begin{align*}
\tilde b(t,x):=b(t,x+\theta_t)-b(t,\theta_t).
\end{align*}
It is easy to see that $\tilde u$ satisfies the following equation:
\begin{align*}
\p_t \tilde u=\sL^{(\alpha)}_{\tilde\kappa_0, \tilde\sigma_0} u+ \tilde b\cdot\nabla  \tilde u+\(\sL^{(\alpha)}_{\tilde\kappa,\tilde\sigma}-\sL^{(\alpha)}_{ \tilde\kappa_0,\tilde\sigma_0}\) \tilde u+ \tilde f,
\end{align*}
where (see \eqref{LL})
$$
\sL^{(\alpha)}_{ \tilde\kappa_0,\tilde\sigma_0} u(t,x):=\int_{\mR^d}\Xi^{(\alpha)} u(t,x\,;\tilde\sigma_0 z)  \tilde\kappa_0(t,z)\nu^{(\alpha)}(\dif z)
$$
with
\begin{align*}
\Xi^{(\alpha)} u(t,x\,;\tilde\sigma_0 z):=u(t,x+\tilde\sigma_0 z)-u(t,x)-\tilde\sigma_0 z^{(\alpha)}\cdot \nabla u(t,x).
\end{align*}
Under {\bf (H$^{\beta}_{\kappa}$)} , {\bf (H$^{\gamma}_{\sigma}$)} and $b\in\mL^\infty_{loc}(\bC^\beta)$, we have
\begin{align}\label{AM}
|\tilde b (t,x)|+|\tilde \kappa(t,x,z)-\tilde \kappa_0(t,z)|\lesssim|x|^\beta\ \ \hbox{and} \ \ |\tilde \sigma(t,x)-\tilde \sigma_0(t)|\lesssim|x|^\gamma.
\end{align}
Taking $\kappa(t,z)=\tilde{\kappa_0}(t,z)$ and $\sigma(t)=\tilde{\sigma_0}(t)$ in \eqref{X-def} and obesrving that \eqref{con0} and \eqref{con1} are still valid in this case, we have a smooth density $p_{s,t}$ for $X_{s,t}^{\tilde{\kappa_0},\tilde{\sigma_0}}$. Define
\begin{align*}
 P_{s,t} f(s,x):=\mE f(s,x+X_{s,t}^{\tilde{\kappa_0},\tilde{\sigma_0}})=\int_{\mR^d} f(s,x+y)p_{s,t}(y) \dif y,\quad \forall x\in\mR^d.
 \end{align*}
Then, by Duhamel's formula \cite{Ch-Ha-Zh}*{Lemma 3.1} we have
\begin{align}\label{DH}
\tilde u(t,x)
=&\int^t_0 P_{s,t}\(\sL^{(\alpha)}_{ \tilde\kappa, \tilde\sigma}-\sL^{(\alpha)}_{\tilde \kappa_0, \tilde\sigma_0}\)  \tilde u(s,x) \dif s
+\int^t_0 P_{s,t}( \tilde b\cdot\nabla \tilde  u)(s,x)\dif s+\int^t_0 P_{s,t} \tilde f(s,x)\dif s.
\end{align}
Below, without loss of generality, we drop the tilde over $ u,\kappa,\kappa_0,\sigma,\sigma_0,b$ and $ f $.
\vspace{0.5em}

We prepare the folowing lemmas which are analogues of \cite{Ha-Wu-Zh}*{Lemma 6.6, 6.8, 6.9}.

\begin{lemma}\label{LeC06}
    Let $ \alpha\in(0,2) $, $ \gamma\in(0,1]$ and $\beta\in(0,(\alpha\wedge1)\gamma)$. Under conditions {\bf (H$^{\beta}_\kappa$)} and {\bf (H$^{\gamma}_{\sigma}$)}, for any $ T>0 $,  there is a constant $ c>0 $ and $\eps\in(0,\beta)$ such that for all $ j\in \mN $, $ t\in [0,T] $ and $u\in\mL^\infty_T (\bC^{\alpha+\eps})$,
    \begin{align}\label{BM01}
    \int^t_0|\Delta_j P_{s,t}\(\sL^{(\alpha)}_{ \kappa, \sigma}-\sL^{(\alpha)}_{   \kappa_0, \sigma_0}\) u|(s,0)\dif s
    \leq c 2^{-(\alpha+\beta)j}\|u\|_{\mL^\infty_T(\bC^{\alpha+\eps})}.
    \end{align}
\end{lemma}

\begin{proof}
    First of all, by defnitions, we have
    \begin{align*}
    \left |\(\sL^{(\alpha)}_{ \kappa, \sigma}-\sL^{(\alpha)}_{   \kappa_0, \sigma_0}\) u(s,x) \right |
    &\leq \left| (\sL^{(\alpha)}_{\kappa,\sigma}-\sL_{\kappa_0,\sigma}^{(\alpha)} ) u(s,x)\right|
    +\left|(\sL_{\kappa_0,\sigma}^{(\alpha)}-\sL_{\kappa_0,\sigma_0}^{(\alpha)}) u(s,x)\right|\\
    & :=J_1+J_2 .
    \end{align*}
  For simplicity of notation, we drop the time variable $t$ and the superscript $ \alpha$ of $ \nu^{(\alpha)} $. For any $\eps>0$, by \eqref{z} and \eqref{AM}, we obtain that for $u\in\bC^{\alpha+\eps}$,
    \begin{align*}
    J_1
    &=\left|\int_{\mR^d} \Xi^{(\alpha)} u(x\,;\sigma z)\cdot (\kappa(x,z)-\kappa_0(z))\nu(\dif z)\right|\\
    &\lesssim |x|^{\beta} \int_{\mR^d}\big|\Xi^{(\alpha)} u(x\,;\sigma z)\big|\nu(\dif z)
    \lesssim |x|^{\beta} \|u\|_{\bC^{\alpha+\eps}}, 
    \end{align*}
    where provided $ |\Xi^{(\alpha)} f(x;z)|\leq \|f\|_{\bC^{\alpha+\eps}}(|z|^{\alpha+\eps}\wedge 1) $. Therefore, by \eqref{CruF}, we have
    \begin{align*}
	\int_0^t|\Delta_jP_{s,t}J_1|(s,0)\dif s
	\lesssim \|u\|_{\bC^{\alpha+\eps}}\int_0^t\int_{\mR^d}|x|^\beta|\Delta_jp_{s,t}(x)|\dif x\dif s
	\lesssim2^{-j(\alpha+\beta)}\|u\|_{\bC^{\alpha+\eps}}.
	\end{align*}
	Next, we estimate $J_2$ for $\alpha\in(0,1)$, $\alpha\in(1,2)$, and $\alpha=1$ separately. 
		
	(1) {\bf Case: $\alpha\in(0,1)$}.
	Choosing $\eps\in(0,\beta)$ such that $\beta<(\alpha-\eps)\gamma$ and $\alpha+\eps<1$, by \eqref{AM}, we have
  \begin{align*}
    J_2
    &\leq \|\kappa\|_{\infty}  \(\int_{|z|\leq 1}+\int_{|z|>1}\) |u(x+\sigma(x)z)-u(x+\sigma(0)z)| \nu^{(\alpha)}(\dif z)\\
    &\lesssim \|u\|_{\bC^{\alpha+\eps}} \([\sigma]_{\bC^\gamma}^{\alpha+\eps}|x|^{(\alpha+\eps)\gamma}\int_{|z|\le1}|z|^{\alpha+\eps}\nu(\dif z)\)
    +\|u\|_{\bC^{\alpha-\eps}}\([\sigma]_{\bC^\gamma}^{\alpha-\eps}|x|^{(\alpha-\eps)\gamma}\int_{|z|>1}|z|^{\alpha-\eps}\nu(\dif z)\).
    \end{align*}
	Hence, by \eqref{CruF},  we get that for all $u\in\bC^{\alpha+\eps}$,
	\begin{align*}
	\int_0^t|\Delta_jP_{s,t}J_2|(s,0)\dif s
	&\lesssim \|u\|_{\bC^{\alpha+\eps}} \int_0^t\int_{\mR^d}(|x|^{(\alpha+\eps)\gamma}+|x|^{(\alpha-\eps)\gamma})|\Delta_jp_{s,t}(x)|\dif x\dif s\\
	&\lesssim 2^{-j\alpha} (2^{-j(\alpha+\eps)\gamma}+2^{-j(\alpha-\eps)\gamma})\|u\|_{\bC^{\alpha+\eps}}
	\lesssim 2^{-j(\alpha+\beta)}\|u\|_{\bC^{\alpha+\eps}}.
	\end{align*}
	
	(2) {\bf Case: $ \alpha\in(1,2) $}. Choosing $\eps\in(0,\beta)$ such that $\alpha+\eps<2$, by\eqref{AM} , we have
	\begin{align*}
    J_2
    &\lesssim  \(\int_{|z|\leq1}  +\int_{|z|>1}\)|(\sigma(x)-\sigma_0)z|\int_0^1|\nabla u(x+r\sigma(x)z+(1-r)\sigma_0z)-\nabla u(x)|\dif r \, \nu(\dif z)\\
    &\lesssim [\sigma]_{\bC^\gamma}|x|^\gamma  \(\|\nabla u\|_{\bC^{\alpha+\eps-1}}\|\sigma\|_{\infty}^{\alpha+\eps-1}\int_{|z|\leq1}|z|^{\alpha+\eps}\nu(\dif z)
    +\|\nabla u\|_{\infty}\int_{|z|>1}|z|\nu(\dif z)\).
    \end{align*}
	Therefore, by \eqref{CruF} and $ \beta<\gamma $, we obtain that for all $ u\in \bC^{\alpha+\eps} $,
	\begin{align*}
	\int_0^t|\Delta_jP_{s,t}J_2|(s,0)\dif s
	&\lesssim \|u\|_{\bC^{\alpha+\eps}}	\int_0^t\int_{\mR^d}|x|^\gamma|\Delta_jp_{s,t}(x)|\dif x\dif s
	\lesssim 2^{-j(\alpha+\beta)}\|u\|_{\bC^{\alpha+\eps}}.
	\end{align*}

	(3) {\bf Case: $ \alpha=1 $}. As above proofs, we decompose the integral on $\mR^d$ into two parts, the small jump part and the large jump part. The small jump part, that is the integral on $\{z\in \mR^d\mid|z|\le1\}$, is same as the case of $\alpha\in(1,2)$. The large jump part, that is  the integral on $\{z\in \mR^d \mid |z|>1\}$ is same as the case of $\alpha\in(0,1)$. Hence, it is easy to see that
	\begin{align*}
	\int_0^t|\Delta_jP_{s,t}J_2|(s,0)\dif s\lesssim 2^{-j(\alpha+\beta)}\|u\|_{\bC^{\alpha+\eps}}.
	\end{align*}
    Combining $J_1$ with $J_2$, we complete the proof.
\end{proof}
\begin{lemma}\label{LeC07}
    Let $ \alpha\in(1/2,2)  $
    and $ \beta\in((1-\alpha)\vee0,\alpha\wedge 1) $. Under the condition $b\in\mL^\infty_{loc}(\bC^\beta)$, for $ T>0 $ and $ \eps\in(0,\alpha+\beta-1) $, there is a constant $c>0$ such that for all $j\in \mN$, $t\in[0,T]$ and $u\in \mL^\infty_T ( \bC^{\alpha+\beta-\eps} )$,
    \begin{align}\label{BM02}
    \int^t_0|\Delta_j P _{s,t}(b\cdot\nabla u)|(s,0)\dif s
    &\leq c 2^{-(\alpha+\beta)j}\|u\|_{\mL^\infty_T(\bC^{\alpha+\beta-\eps})}.
    \end{align}
\end{lemma}

\begin{proof}
  By the \eqref{AM} and \eqref{CruF}, we have	
    \begin{align*}
    \int^t_0|\Delta_j P _{s,t}(b\cdot\nabla u)|(s,0)\dif s 
    &=\int^t_0 \left| \int_{\mR^d}\Delta_j p_{s,t}(x)\cdot (b\cdot\nabla u)(s,x)\dif x\right| \dif s\\
   &\lesssim\int_0^t \|\nabla u(s)\|_\infty \int_{\mR^d} |x|^\beta |\Delta_j p_{s,t}(x)| \dif x \dif s\lesssim 2^{-(\alpha+\beta)j} \|u\|_{\mL^\infty_T(\bC^{\alpha+\beta-\eps})},
    \end{align*}
    where we used the fact that $ 1<\alpha+\beta-\eps $.   
 \end{proof}

\begin{lemma}\label{LeC08}
    Let $ \alpha\in(0,2)  $ and $\beta\in\mR_+$. For any $T>0$, there is a constant $c>0$ such that for all $j\in \mN$,  $t\in[0,T]$ and $f \in \mL^\infty_T ( \bC^{\beta})$,
    \begin{align}\label{BM03}
    \int^t_0|\Delta_j P_{s,t}f|(s,0)\dif s
    &\leq c 2^{-(\alpha+\beta)j}\|f\|_{ \mL^\infty_T ( \bC^{\beta})}.
    \end{align}
\end{lemma}

\begin{proof}
    By \eqref{DE}, \eqref{CruF} and \autoref{DF}, we have
    \begin{align*}
    \int^t_0|\Delta_j P _{s,t} f|(s,0)\dif s 
    &=\int^t_0 \left| \int_{\mR^d}\widetilde\Delta_j p_{s,t}(x)\cdot \Delta_j f(s,x)\dif x\right| \dif s\\
    &\lesssim \int_{0}^t \|\Delta_j f(s)\|_{\infty}  \int_{\mR^d} | \widetilde\Delta_j p_{s,t}(x)|\dif x \dif s
    \lesssim 2^{-(\alpha+\beta) j}\|f\|_{ \mL^\infty_T ( \bC^{\beta})}.
    \end{align*} 
    Thus, we get \eqref{BM03}.
\end{proof}

Now we are in a position to give

\begin{proof}[Proof of \autoref{CO}]
    By \eqref{DH}, \autoref{LeC06}, \autoref{LeC07} and \autoref{LeC08}, we have
    \begin{align*}
    |\Delta_j u(t,\theta_t)|=|\Delta_j \tilde u(t,0)|\lesssim 2^{-(\alpha+\beta)j} \( \|u\|_{\mL^\infty_T(\bC^{\alpha+\beta-\eps})}+\|f\|_{\mL^\infty_T(\bC^{\beta})}\),
    \end{align*}
    for some $ \eps\in(0,\beta\wedge(\alpha+\beta-1)) $. By taking the supremum of $ x_0 $ and \autoref{HWZ}, we obatin
    \begin{align}\label{eq410}
    \|\Delta_j u(t)\|_\infty \lesssim  2^{-(\alpha+\beta)j} \( \|u\|_{\mL^\infty_T(\bC^{\alpha+\beta-\eps})}+\|f\|_{\mL^\infty_T(\bC^{\beta})}\),
    \end{align}
    which in turn implies \eqref{co} by the interpolation inequality $\|u\|_{\mL^\infty_T(\bC^{\alpha+\beta-\eps})}\le \eps\|u\|_{\mL^\infty_T(\bC^{\alpha+\beta})}+c_\eps\|u\|_{\mL^\infty_T}$ for any $\eps\in(0,1)$ and some constant $c_\eps$, the maximum principle \autoref{MM}, and \autoref{DF}. The proof is completed.
 \end{proof}

\begin{remark}\label{R38}
The restriction of $\alpha\in(1/2,2)$ is only used in \autoref{LeC07}, which is caused by the moment problem due to $1-\alpha<\alpha$. Since we consider classical solutions, $\alpha+\beta$ must be larger than $1$ so that $\nabla u$ is meaningful. In addition, we  shall  assume $\beta<\alpha$ due to the moment estimate (see \autoref{CRU}). The critical case $\alpha+\beta=1$ is a technical problem, and we have no ideas to fix it.
\end{remark}

\subsection{Unbounded drift case}\label{BS}
In this subsection, we use a cutoff technique depending on characterization curve making unbounded drift bounded to prove \autoref{SFA}. We first establish a commutator estimate. 

\begin{lemma}\label{COM}
	Let $\alpha\in(0,2)$, $\gamma\in(0,1]$ and $\beta\in(0,(\alpha\wedge 1)\gamma)$. Under conditions {\bf (H$^{\beta}_\kappa$)} and {\bf (H$^{\gamma}_{\sigma}$)}, for any $T>0$, there is a constant $c>0$ such that for any $u\in\mL^\infty_T(\bC^\alpha)$,
	\begin{align*}
	\left \| [\chi,\sL^{(\alpha)}_{\kappa,\sigma}]u \right \|_{\mL^\infty_T(\bC^\beta)} \leq c\|u\|_{\mL^\infty_T(\bC^\alpha)}.
	\end{align*}
\end{lemma} 
The definition of the notation $[\cdot,\cdot]$ can be found at the end of introduction and $\chi$ is defined by \eqref{CHI}.
\begin{proof}
	Rewrite
	\begin{align}\label{W-17A}
	[\chi,\sL^{(\alpha)}_{\kappa,\sigma}]u(t,x)
    =\int_{\mR^d}\Sigma(t,x,z)\kappa(t,x,z)\nu^{(\alpha)}(\dif z),
	\end{align}
where
\begin{align*}
\Sigma(t,x,z):=\delta_{\sigma(t,x)z}\chi(x)u(t,x+\sigma(t,x)z)-u(t,x)\sigma(t,x)z^{(\alpha)}\cdot\nabla\chi(x)
\end{align*}
 with $ z^{(\alpha)}:=z\1_{\alpha\in(1,2)}+z\1_{|z|\leq1} \1_{\alpha=1}$ and the definition of the notation $\delta_hf$ is defined in the beginning of Subsection \ref{BHs}. For simplicity of notation, we drop the time variable $t$ and the superscript $ \alpha $ of $ \nu^{(\alpha)} $. As the proof of \autoref{LeC06}, we split the integral \eqref{W-17A} over areas $\{z\in \mR^d \mid |z|\le1\}$ and $\{z\in \mR^d \mid |z|>1\}$, that is  $[\chi,\sL^{(\alpha)}_{\kappa,\sigma}]u(t,x)=J_1+J_2$ with
\begin{align*}
J_1(x):=\int_{|z|\le1}\Sigma(x,z)\kappa(x,z)\nu(\dif z)\ \ \hbox{and}\ \   J_2(x):=\int_{|z|>1}\Sigma(x,z)\kappa(x,z)\nu(\dif z),
\end{align*}
where $J_1$ and $J_2$ are called the small jump part and the large jump part, respectively. Since $z^{(\alpha)}$ has different forms in cases $\alpha<1$, $\alpha>1$ and $\alpha=1$, we estimate $J_1,J_2$ for these cases separately.	
Here, the key to estimate $\|\cdot\|_{\bC^\beta}$ norms of those integrals is the following fact
 \begin{align*}
\|\int g(\cdot,z)\dif z\|_{\bC^\beta}\le\int\|g(\cdot,z)\|_{\bC^\beta}\dif z.
\end{align*}
\begin{enumerate}[(1)]
\item {\bf Case: $ \alpha\in(0,1) $}. By \autoref{FuHe}, {\bf (H$^{\gamma}_\sigma$)}, and the fact $\|fg\|_{\bC^\beta}\le\|f\|_{\bC^\beta}\|g\|_{\bC^\beta}$, we derive that
\begin{align}\label{BA}
\begin{split}
\|\Sigma(\cdot,z)\|_{\bC^\beta}\lesssim (|z|\wedge|z|^{\beta/\gamma}) \|u\|_{\bC^\alpha},
\end{split}
\end{align}
where we used 
\begin{align}\label{SJE}
\Sigma(x,z)=(z\cdot\int_0^1\nabla\chi(x+s\sigma(x)z)\dif s\)u(x+\sigma(x)z)
\end{align}
for $|z|\leq 1$. Therefore, by {\bf (H$^{\beta}_{\kappa}$)} and \eqref{z} with $0<\beta/\gamma< \alpha<1$,  we obtain
\begin{align*}
\|J_1\|_{\bC^\beta}+\|J_2\|_{\bC^\beta}
\lesssim \|u\|_{\bC^\alpha} \int_{\mR^d}|z|\wedge|z|^{\beta/\gamma}\nu (\dif z)
\lesssim \|u\|_{\bC^\alpha},
\end{align*}
which in turn gives the desired result.
\item {\bf Case: $\alpha\in(1,2)$}. For the large jump part $J_2$,  rewrite 
\begin{align}\label{SJE}
\Sigma(x,z)=(z\cdot\int_0^1\nabla\chi(x+s\sigma(x)z)\dif s\)u(x+\sigma(x)z)-u(x)\sigma(x)z\cdot\nabla\chi(x).
\end{align}
Noticing $0< \beta/\gamma< 1<\alpha $, by \autoref{FuHe}, {\bf (H$^{\beta}_\kappa$)}, and {\bf (H$^{\gamma}_\sigma$)},  we have
$$
\|J_2\|_{\bC^\beta}
\lesssim\|u\|_{\bC^\alpha}\int_{|z|>1}(|z|^{\beta/\gamma}+|z|)\nu^{(\alpha)}(\dif z)\lesssim \|u\|_{\bC^\alpha}.
$$
For $J_1$, we need the interpolation inequality \autoref{INTER}. Let 
$$
\cT_z u(x):=u(x+\sigma(x)z)-u(x).
$$
By \autoref{FuHe}, we have,
\begin{align*}
\|\cT_z u\|_{\bC^\beta}=\|z\cdot\int_0^1\nabla u(\cdot+r\sigma(\cdot)z)\dif r\|_{\bC^\beta}
\lesssim\|u\|_{\bC^{1+\beta/\gamma}}|z|.
\end{align*}
We also have		
\begin{align*}
\|\cT_z u\|_{\bC^\beta}\leq \| u(\cdot+\sigma(\cdot)z)\|_{\bC^\beta}+\|u\|_{\bC^\beta} \lesssim  \|u\|_{\bC^{\beta/\gamma}}.
\end{align*}
Choose some $\theta\in(\alpha-1,\alpha-\beta/\gamma)$ such that $\vartheta:=\theta(1+\frac{\beta}{\gamma})+(1-\theta)\frac{\beta}{\gamma}<\alpha$. Then, by \autoref{INTER}, we obtain that
\begin{align}\label{BD}
\|\cT_z u\|_{\bC^\beta}\lesssim |z|^\theta\|u\|_{\bC^\vartheta}\le|z|^\theta\|u\|_{\bC^\alpha}.
\end{align}
Since
\begin{align*}		
\Sigma(x,z)&=\(\sigma(x)z\cdot\int_0^1 \delta_{r\sigma(x)z}\nabla \chi(x) \dif r\) u(x+\sigma(x)z)+\sigma(x)z\cdot\cT_z u(x),
\end{align*}
by \eqref{BD} and \autoref{FuHe}, we obtain that 
\begin{align*}
\|J_1\|_{\bC^\beta}
&\lesssim \int_{|z|\le1}\|(\chi(\cdot+\sigma(\cdot)z)-\chi)u(\cdot+\sigma(\cdot)z)-u\sigma z\cdot\nabla\chi\|_{\bC^\beta}\nu^{(\alpha)}(\dif z)\\
&\lesssim \|u\|_{\bC^\alpha} \int_{|z|\le1}(|z|^2+|z|^{1+\theta})\nu^{(\alpha)}(\dif z)
\lesssim\|u\|_{\bC^\alpha}.
\end{align*} 
\item {\bf Case: $ \alpha=1$}. Observe that for the case $ \alpha=1 $, $J_1$ is same as the case of $\alpha\in(1,2)$ and $J_2$ is same as the case of $\alpha\in(0,1)$.
\end{enumerate} 
Combining the above calcultions, we complete the proof.
\end{proof}	

Now we are in a position to give

\begin{proof}[Proof of \autoref{SFA}]
    Fix $x_0\in\mR^{d}$. Let $\theta_t$, $
    \tilde u$, $\tilde f$, $\tilde \sigma$, 
    $\tilde \kappa$, $\tilde \sigma_0$, $\tilde \kappa_0$  and $\tilde b$
 be the same ones in Subsection \ref{UnS}.  See that $ \tilde \kappa $ and $ \tilde \sigma $ still satisfy {\bf (H$^{\beta}_{\kappa}$)} and {\bf (H$^{\gamma}_{\sigma}$)} respectively. The only difference we shall note is that $b\notin\mL^\infty_{loc}(\bC^\beta)$ here. We use the cutoff technique to fix this problem below. By \eqref{FE}, it is easy to see that $\tilde u$ satisfies the following equation:
    \begin{align}\label{EM0}
    \p_t \tilde u=\sL^{(\alpha)}_{\tilde \kappa,\tilde \sigma} \tilde u+ \tilde b\cdot\nabla \tilde u+\tilde f.
    \end{align}
Observe that
    \begin{align}\label{AM04}
    \p_t (\chi \tilde u)=\sL^{(\alpha)}_{\tilde  \kappa, \tilde \sigma} (\chi \tilde  u)+(\chi   \tilde b)\cdot \nabla (\chi \tilde u)+\chi  \tilde f+[\chi,\sL^{(\alpha)}_{\tilde \kappa,\tilde \sigma}] \tilde  u-(\chi \tilde b)\tilde u\cdot \nabla\chi,
    \end{align}
    where $ \chi $ is defined by \eqref{CHI}. Moreover, by \autoref{CutF} and {\bf (H$^{\beta}_{b}$)}, we have $ \chi \tilde b(t,\cdot) \in \bC^\beta $ and 
    \begin{align}\label{AM03}
    \|(\chi \tilde  b) (t,\cdot)\|_{\bC^\beta}\lesssim [b(t,\cdot)]_{\mC^\beta}\leq c_0.
    \end{align}
    Thus, concluding form \autoref{CO}, \eqref{AM03}, and \autoref{COM}, for any $t\le T$, we obtain that
    \begin{align*}
    \begin{split}
    \|\chi \tilde u(t,\cdot)\|_{\bC^{\alpha+\beta}}
    &\lesssim \|\chi \tilde f\|_{\mL_T^\infty(\bC^{\beta})}+\|[\chi,\sL^{(\alpha)}_{\tilde \kappa,\tilde \sigma}] \tilde  u\|_{\mL_T^\infty(\bC^{\beta})}+\|(\chi \tilde b)\tilde u\cdot \nabla\chi\|_{\mL_T^\infty(\bC^{\beta})}\\
    &\lesssim \|f\|_{\mL_T^\infty(\bC^{\beta})}+\|u\|_{\mL_T^\infty(\bC^{\alpha})}+\|u\|_{\mL_T^\infty(\bC^{\beta})}.
    \end{split}
    \end{align*}   
    Noticing that, for any $ k\in \mN_0 $, 
    $$
    \nabla^k (\chi \tilde u)(t,x) =  \nabla^k \tilde u(t,x),\ \ \forall |x|\leq 1/4, \,  t\in \mR_+,
    $$
and for any $t>0$, 
    \begin{align*} 
  \|u(t,\cdot ) \|_{\bC^{\alpha+\beta}(B(\theta_t,1/4))}\leq \|(\chi \tilde u)(t,\cdot)\|_{\bC^{\alpha+\beta}}
    \end{align*} 
with $B(\theta_t,1/4) = \{x\in \mR^d \mid  |x- \theta_t|\leq 1/4\}$.
Therefore, by \autoref{HB}, \autoref{HWZ}, and taking supremum of $x_0$, we get
    \begin{align*}
    \|u\|_{\mL_T^\infty(\bC^{\alpha+\beta})}
    \lesssim \|f\|_{\mL_T^\infty(\bC^{\beta})}+\|u\|_{\mL_T^\infty(\bC^{\alpha})}+\|u\|_{\mL_T^\infty(\bC^{\beta})}.
    \end{align*} 
    Furthermore, by interpolations and the maximum principle \autoref{MM}, we have
    \begin{align*}
    \|u\|_{\mL^\infty_T(\bC^{\alpha+\beta})}\lesssim \|f\|_{\mL^\infty_T(\bC^\beta)}.
    \end{align*}
    The proof is completed.
\end{proof}

\begin{remark}\label{R310}
Observe that, by \autoref{HB},
$$
|\tilde b (t,x)|\leq [b(t,\cdot)]_{\mC^\beta} (|x|^{\beta}\1_{\{|x|\leq 1\}}+|x|\1_{\{|x|>1\}}).
$$
Hence, if we have the following estimate  
\begin{align}\label{DR}
\int_0^t\int_{\mR^d}|x||\nabla\Delta_jp_{s,t}(x)|\dif x\dif s\lesssim 2^{-j\alpha},
\end{align}
 then we get \autoref{LeC07} under  {\bf (H$^{\beta}_{b}$)} with $\beta\in((1-\alpha)\vee0,\gamma)$. Furthermore, we get \autoref{SEA} directly without using cutoff techiques as showned in Subsection \ref{UnS}. Fortunately, if $\alpha \in (1,2)$ or $\sL^{(\alpha)}_{\kappa,\sigma}=\Delta^{\alpha/2}$ with $\alpha\in(0,1)$, \eqref{DR} is true.
\end{remark}


\section{Proof of \autoref{SEA}}\label{EX}

 Since we have the a priori estimate \autoref{SFA}, the rest of proof for  \autoref{SEA} is existence of the classical solution.  We first  introduce the following useful lemma.

\bl\label{SS}
Let $\alpha\in(0,2)$ and $\beta\in((1-\alpha)\vee0,1\wedge\alpha)$ with $\alpha+\beta\notin\mN$. Under conditions {\bf (H$^\beta_\kappa$)} and  {\bf (H$^1_\sigma$)}, there is a constant $c>0$ such that for any $T>0$ and $u\in \mL_{loc}^\infty(\bC^{\alpha+\beta})$,
\begin{align*}
\|\sL_{\kappa,\sigma}^{(\alpha)} u\|_{\mL^\infty_T(\bC^\beta)}\le c\|u\|_{\mL^\infty_T(\bC^{\alpha+\beta})}.
\end{align*}
\el

\bpf
For the simplicity, we only consider the case $0<\alpha<1$, and drop the time variable $t$ and the superscript $\alpha$ of $\nu^{(\alpha)}$ and $\sL_{\kappa,\sigma}^{(\alpha)}$. In this case, by \eqref{LL},
\begin{align}
\sL_{\kappa,\sigma}u(x)=\int_{\mR^d}(u(x+\sigma(x)z)-u(t,x))\kappa(x,z)\nu(\dif z).
\end{align}
For any fixed $x_0,h\in\mR^d$, define
$$
\kappa_0(x,z):\equiv\kappa(x_0,z),\quad \sigma_0(x):\equiv\sigma(x_0).
$$
 Notice that
\begin{align*}
\sL_{\kappa,\sigma}u(x_0+h)-\sL_{\kappa,\sigma}u(x_0)
&=\sL_{\kappa,\sigma}u(x_0+h)-\sL_{\kappa_0,\sigma}u(x_0+h)\\
&+\sL_{\kappa_0,\sigma}u(x_0+h)-\sL_{\kappa_0,\sigma_0}u(x_0+h)\\
&+\sL_{\kappa_0,\sigma_0}u(x_0+h)-\sL_{\kappa,\sigma}u(x_0)\\
&:=\sI_1+\sI_2+\sI_3.
\end{align*}
For $\sI_1$, under {\bf (H$^\beta_\kappa$)} and {\bf (H$^1_\sigma$)}, by \eqref{z}, we have
\begin{align*}
|\sI_1|\lesssim \|\kappa\|_{\bC^\beta}|h|^\beta \int_{\mR^d}\[(\|\sigma\|_{L^\infty}\|u\|_{\bC^1}|z|)\wedge\|u\|_{L^\infty}\]\nu(\dif z)\lesssim\|u\|_{\bC^{\alpha+\beta}}|h|^\beta.
\end{align*}
For $\sI_2$, under {\bf (H$^1_\sigma$)} and $0< \beta<\alpha<1$, we have
 \begin{align*}
|u(x_0+h+\sigma(x_0+h)z)-u(x_0+h+\sigma(x_0)z)|
\leq |h|^\beta\[ (\|u\|_{\bC^1}|z|)\wedge(\|u\|_{\bC^\beta}|z|^\beta)\],
\end{align*}
Hence, by \eqref{z} and {\bf (H$^\beta_\kappa$)}, we get
\begin{align*}
|\sI_2|\lesssim \|u\|_{\bC^1}|h|^\beta \int_{\mR^d}(|z|\wedge|z|^\beta)\nu(\dif z)\lesssim \|u\|_{\bC^{\alpha+\beta}}|h|^\beta.
\end{align*}
For $\sI_3$, we use the block operator $\Delta_j$. Define 
$$
\sL_0 u(x):=\int_{\mR^d}\(u(x_0+x+\sigma(x_0)z)-u(x_0+x)\)\kappa(x_0)\nu(\dif z).
$$
Note that $|f(x+h)-f(x)|\leq (\|\nabla f\|_{L^\infty}|z|)\wedge (2\|f\|_\infty)$. For any $j\ge0$, by Bernstein's inequality \eqref{Bernstein} and $\nu(\dif \lambda z) = \lambda^\alpha \nu(\dif z)$ with $\lambda>0$, we have
 \begin{align*}
|\Delta_j\sL_0 u(x)|&=\left|\int_{\mR^d}\(\Delta_j u(x_0+x+\sigma(x_0)z)-\Delta_ju(x_0+x)\)\kappa(x_0)\nu(\dif z)\right|\\
&\lesssim \int_{\mR^d}(|\sigma(x_0)z|\|\nabla\Delta_j u\|_{L^\infty})\wedge\|\Delta_j u\|_{L^\infty}\|\kappa\|_{L^\infty}\nu(\dif z)\\
&\lesssim  \|\Delta_j u\|_{L^\infty} \int_{\mR^d}(|2^jz|\wedge1)\nu(\dif z)\lesssim 2^{\alpha j}\|\Delta_j u\|_{L^\infty} ,
\end{align*}
which implies that $\|\sL_0 u\|_{\bC^{\beta}}\lesssim \|u\|_{\alpha+\beta}$. Thus, we have
 \begin{align*}
|\sI_3|=|\sL_0 u(h)-\sL_0 u(0)|\lesssim |h|^\beta\|u\|_{\alpha+\beta},
\end{align*}
which completes the proof.
\epf

Now, we are in a position to give

\begin{proof}[Proof of \autoref{SEA}]
\noindent
\begin{enumerate}
\item[({\it Step 1})]
Suppose that $b$ is bounded and $\sigma$ is Lipschitz in this step. Consider the following continuity equation:
 \begin{align}\label{ConE}
\p_tu=\lambda \sL_{1,\sigma}^{(\alpha)} u+(1-\lambda)\sL_{\kappa,\sigma}^{(\alpha)} u+b\cdot\nabla u+f,\quad u(0)=0.
\end{align}
When $\lambda=1$, by the same argument  as in \cite{Ch-Ha-Zh}*{Section 5} and \autoref{SFA}, there is a unique classical solution for Eq.\eqref{ConE}. Using the continuity method, by \autoref{SS} and \autoref{SFA}, we get a classical solution $u$ for Eq.\eqref{ConE} with $\lambda=0$ that is PDE \eqref{in-PDE}.
\item[({\it Step 2})]
For any $n\in\mN$ and $(t,x)\in\mR_+\times\mR^d$, let $b_n(t,x):=b(t,x)\wedge n$ and $\sigma_n(t,x):=\sigma(t)*\rho_{n}(x)$ where $\rho_n$ is the usual modifier. By Step 1, there is a classical solution $u_n$ of PDE \eqref{in-PDE} with $b=b_n$ and $\sigma=\sigma_n$, i.e.,
\begin{align}\label{F1.2}
\p_t u_n=\sL_{\kappa,\sigma_n}^{(\alpha)} u_n +b_n\cdot\nabla u_n+f,\quad u_n(0)=0.
\end{align}
Noting that $[b_n(t)]_{\mC^\beta}\le[b(t)]_{\mC^\beta}$ and $\|\sigma_n(t)\|_{\bC^\gamma}\le\|\sigma(t)\|_{\bC^\gamma}$, by \autoref{SFA}, there is a constant $c$ such that for all $n\in\mN$,
\begin{align}\label{W-K18}
\|u_n\|_{\mL^\infty_T(\bC^{\alpha+\beta})}\le  c\|f\|_{\mL^\infty_T(\bC^\beta)}.
\end{align}
Moreover, by \autoref{SS}, we have
\begin{align}\label{W-K01}
\sup_n\|\sL^{(\alpha)}_{\kappa,\sigma_n}u_n\|_{\mL_T^\infty}\lesssim\sup_n\|u_n\|_{\mL^\infty_T(\bC^{\alpha+\beta})}\leq c\|f\|_{\mL^\infty_T(\bC^{\beta})}.
\end{align} 
Thus, combining \eqref{F1.2} with the above inequality, for any $M>0$ and $s,t\in [0,T]$,
\begin{align*}
 \|u_n(t)-u_n(s)\|_{L^\infty(B_M)}&\lesssim |t-s| (1+\|b\|_{\mL^\infty_T(B_M)}) \|f\|_{\mL^\infty_T(\bC^\beta)}\\
 &\lesssim|t-s| \(1+\sup_{t\in\mR_+}|b(t,0)|M\) \|f\|_{\mL^\infty_T(\bC^\beta)}.
\end{align*}
Therefore, by Ascolli-Arzela's theorem, there is a function $u\in \cap_{M>0}C(\mR_+;\bC^{(\alpha\vee1)+\eps}(B_M))$ and a  subsequence $ \{n_k\}_{k\geq 1} $ such that for all $t\in[0,T]$ and $M>0$,
\begin{align}\label{Asco}
\lim_{j\to +\infty}\sup_{|x|\leq M}|\nabla^m u_{n_j}(t,x)-\nabla^m u(t,x)|=0,\ \  m=0,1.
\end{align}
For the convenience, we drop the subscript $ k $ of $n_k$. By definitions and \eqref{W-K18}, we have
\begin{align*}
\|u\|_{\mL_T^\infty(\bC^{\alpha+\beta})}&=\sup_{t\in[0,T]}\sup_{j\ge0}2^{(\alpha+\beta)j}\|\Delta_j u(t)\|_{L^\infty}=\sup_{t\in[0,T]}\sup_{j\ge0}2^{(\alpha+\beta)j}\|\lim_{n\to\infty}\Delta_ju_n(t)\|_{L^\infty}\\
&\le\sup_{t\in[0,T]}\sup_{j\ge0}2^{(\alpha+\beta)j}\sup_{n}\|\Delta_ju_n(t)\|_{L^\infty}\le\sup_n\|u_n\|_{\mL_T^\infty(\bC^{\alpha+\beta})}\le c\|f\|_{\mL_T^\infty(\bC^{\beta})}.
\end{align*}
Notice that
\begin{align}\label{last}
u_n(t,x)=\int_0^t\(\sL^{(\alpha)}_{\kappa,\sigma_n}u_n(s,x)+b_n(s,x)\cdot\nabla u_n(s,x)+f(s,x)\)\dif s
\end{align}
with $\sup_n\|\sL^{(\alpha)}_{\kappa,\sigma_n}u_n\|_{\mL_T^\infty} \lesssim \|f\|_{\mL^\infty_T(\bC^{\beta})}$ and 
$$
\sup_{t\in\mR_+,n\in\mN}|b_n(t,x)|\le\sup_{t\in\mR_+}|b(t,x)|\le\sup_{t\in\mR_+}|b(t,0)|+2|x|\[b(t)\]_{\mC^\beta}<+\infty, \forall x\in\mR^d.
$$
Letting $n\to\infty$ in \eqref{last}, by \eqref{Asco} and the dominated convergence theorem  we have
\begin{align*}
u(t,x)=\int_0^t\(\sL^{(\alpha)}_{\kappa,\sigma}u(s,x)+b(s,x)\cdot\nabla u(s,x)+f(s,x)\)\dif s.
\end{align*}
Hence, the function $u$ is a classical solution of PDE \eqref{FE} in the sense of \autoref{DEF}.
\end{enumerate}
The proof is finished.
\end{proof}


\paragraph{\bf Acknowledgements} We are grateful to St\'ephane Menozzi, Xicheng Zhang and Guohuan Zhao for their useful suggestions for this paper.

\end{document}